\documentclass[a4paper,12pt,final]{amsart}
\usepackage{times,a4wide,mathrsfs,amssymb,dsfont}

\newcommand{\C}{\mathbb{C}}
\newcommand{\Z}{\mathbb{Z}}

\newcommand{\QQ}{\mathbb{Q}}

\newcommand{\PP}{\mathbb{P}}

\newcommand{\A}{\mathbb{A}}

\newcommand{\MM}{\mathcal M}

\newcommand{\gr}{\hbox{Gr}}

\newcommand{\ide}{\hbox{id}}

\newtheorem{theorem}{Theorem}[section]

\newtheorem{lemma}[theorem]{Lemma}
\newtheorem{corollary}[theorem]{Corollary}

\newtheorem{nonumbering}{Theorem}

\newtheorem{convention}{Conventions}

\theoremstyle{definition}
\newtheorem{remark}[theorem]{Remark}

\newtheorem{nonumberingt}{Acknowledgements}

\begin{document}
\author[Robert Laterveer]
{Robert Laterveer}

\address{Institut de Recherche Math\'ematique Avanc\'ee,
CNRS -- Universit\'e 
de Strasbourg,\
7 Rue Ren\'e Des\-car\-tes, 67084 Strasbourg CEDEX,
FRANCE.}
\email{robert.laterveer@math.unistra.fr}

\title[On the motive of IMOU CY threefolds]{On the motive of Ito--Miura--Okawa--Ueda Calabi--Yau threefolds}

\begin{abstract} Ito-Miura-Okawa-Ueda have constructed a pair of Calabi--Yau threefolds $X$ and $Y$ that are L-equivalent and derived equivalent, but not stably birational.
We complete the picture by showing that $X$ and $Y$ have isomorphic Chow motives.
\end{abstract}



\maketitle

\section{Introduction}

Let $\hbox{Var}(k)$ denote the category of algebraic varieties over a field $k$.
The Grothendieck ring $K_0(\hbox{Var}(k))$ encodes fundamental properties of the birational geometry of varieties. The intricacy of the ring $K_0(\hbox{Var}(k))$
is highlighted by the result of Borisov \cite{Bor}, showing that the class of the affine line $[\A^1]$ is a zero--divisor in $K_0(\hbox{Var}(k))$. Inspired by \cite{Bor}, Ito--Miura--Okawa--Ueda \cite{IMOU} exhibit a pair of Calabi--Yau threefolds $X, Y$ that are {\em not\/} stably birational (and so $[X]\not=[Y]$ in the Grothendieck ring), but
  \[          ([X] -[Y]) [\A^1]=0\ \ \ \hbox{in}\ K_0(\hbox{Var}(k)) \ \]
  (i.e., $X$ and $Y$ are ``L-equivalent'', a notion studied in \cite{KS}).
  
  As shown by Kuznetsov \cite{Kuz}, the threefolds $X, Y$ of \cite{IMOU} are derived equivalent. 
  According to a conjecture of Orlov \cite[Conjecture 1]{Or}, derived equivalent smooth projective varieties should have isomorphic Chow motives.
  The aim of this tiny note is to check that such is indeed the case for the threefolds $X, Y$:
  
  \begin{nonumbering}[=theorem \ref{main}] Let $X, Y$ be the two Calabi--Yau threefolds of \cite{IMOU}. Then
   \[ h(X)\cong h(Y)\ \ \ \hbox{in}\ \MM_{\rm rat}\ .\]
   \end{nonumbering}
   
 An immediate corollary is that if $k$ is a finite field, then $X$ and $Y$ share the same zeta function (corollary \ref{cor}).

\vskip0.6cm

\begin{convention} In this note, the word {\sl variety\/} will refer to a reduced irreducible scheme of finite type over a field $k$. For a smooth variety $X$, we will denote by $A^j(X)$ the Chow group of codimension $j$ cycles on $X$ 
with $\QQ$-coefficients.

The notation 
$A^j_{hom}(X)$ will be used to indicate the subgroups of 
homologically trivial cycles.
For a morphism between smooth varieties $f\colon X\to Y$, we will write $\Gamma_f\in A^\ast(X\times Y)$ for the graph of $f$, and ${}^t \Gamma_f\in A^\ast(Y\times X)$ for the transpose correspondence.

The contravariant category of Chow motives (i.e., pure motives with respect to rational equivalence as in \cite{Sc}, \cite{MNP}) will be denoted $\MM_{\rm rat}$. 
\end{convention}

 \section{The Calabi--Yau threefolds}
 
 \begin{theorem}[Ito--Miura--Okawa--Ueda \cite{IMOU}]\label{imou} Let $k$ be an algebraically closed field of characteristic $0$. There exist two Calabi--Yau threefolds $X, Y$ over $k$ such that
   \[ [X]\not= [Y]\ \ \ \hbox{in}\ K_0(\hbox{Var}(k)) \ ,\]
   but
   \[  ([X] -[Y]) [\A^1]=0\ \ \ \hbox{in}\ K_0(\hbox{Var}(k)) \ .\]
    \end{theorem}

   \begin{theorem}[Kuznetsov \cite{Kuz}]\label{kuz} Let $k$ be any field. The threefolds $X, Y$ over $k$ constructed as in \cite{IMOU} are derived equivalent: there is an isomorphism between the bounded derived categories of coherent sheaves
     \[ D^b(X) \cong D^b(Y)\ .\]
     
     In particular, if $k=\C$ then there is an isomorphism of polarized Hodge structures
     \[ H^3(X,\Z)\ \cong\ H^3(Y,\Z)\ .\]
     \end{theorem}
     
     \begin{proof} The derived equivalence is \cite[Theorem 5]{Kuz}. The isomorphism of Hodge structures is a corollary of the derived equivalence, in view of
     \cite[Proposition 2.1 and Remark 2.3]{OR}.
      \end{proof}
    
  \begin{remark} The construction of the threefolds $X, Y$ in \cite{IMOU} works over any field $k$. However, the proof that $[X]\not=[Y]$ uses the MRC fibration and is (a priori) restricted to characteristic $0$. The argument of \cite{Kuz}, on the other hand, has no characteristic $0$ assumption.
  \end{remark}

\section{Main result}     

\begin{theorem}\label{main} Let $k$ be any field, and let $X, Y$ be the two Calabi--Yau threefolds over $k$ constructed as in \cite{IMOU}. Then
   \[ h(X)\cong h(Y)\ \ \ \hbox{in}\ \MM_{\rm rat}\ .\]
  \end{theorem} 
     
  \begin{proof} 
  First, to simplify matters, let us slightly cut down the motives of $X$ and $Y$. It is known \cite{IMOU} that $X$ and $Y$ have Picard number $1$. A routine argument  gives a decomposition of the Chow motives
  \[ \begin{split}   h(X)&= \mathds{1} \oplus \mathds{1}(1)\oplus h^3(X) \oplus \mathds{1}(2) \oplus \mathds{1}(3)\ ,\\
                   h(Y)&= \mathds{1} \oplus \mathds{1}(1)\oplus h^3(Y) \oplus \mathds{1}(2) \oplus \mathds{1}(3)\ \ \ \ \ \ \hbox{in}\ \MM_{\rm rat}\ ,\\
                  \end{split} \] 
  where $\mathds{1}$ is the motive of the point $\hbox{Spec}(k)$. (The gist of this ``routine argument'' is as follows: let $H\in A^1(X)$ be a hyperplane section. Then 
    \[ \pi^{2i}_X:= c_i H^{3-i}\times H\ \ \ \in\ A^3(X\times X)\ , \ \ \ 0\le i\le 3\ ,\] 
    defines an orthogonal set of projectors lifting the K\"unneth components, for appropriate $c_i\in\QQ$. One can then define $\pi^3_X=\Delta_X-\sum_i \pi^{2i}_X\in A^3(X\times X)$, and $h^j(X)=(X,\pi^i_X,0)\in \MM_{\rm rat}$, and ditto for $Y$.)
  
  To prove the theorem, it will thus suffice to prove an isomorphism of motives
    \begin{equation}\label{iso3}   h^3(X)\cong h^3(Y)\ \ \ \hbox{in}\ \MM_{\rm rat}\ .\end{equation}
    We observe that the above decomposition (plus the fact that $H^\ast(h^3(X))=H^3(X)$ is odd--dimensional) implies equality
     \[ A^\ast(h^3(X)) = A^\ast_{hom}(X)\ ,\]
     and similarly for $Y$.

  The rest of the proof will consist in finding a correspondence $\Gamma\in A^3(X\times Y)$ inducing isomorphisms
    \begin{equation}\label{isochow} \Gamma_\ast\colon \ \ A^j_{hom}(X_K)\ \xrightarrow{\cong}\ A^j_{hom}(Y_K)\ \ \ \forall j\ ,\end{equation}
    for all field extensions $K\supset k$.
   By the above observation, this means that $\Gamma$ induces isomorphisms
    \[  A^j(h^3(X)_K)\ \xrightarrow{\cong}\ A^j(h^3(Y)_K)\ \ \ \forall j\ ,\]  
    which (as is well-known, cf. for instance \cite[Lemma 1.1]{Huy}) ensures that $\Gamma$ induces the required isomorphism of Chow motives (\ref{iso3}).

To find the correspondence $\Gamma$, we need look no further than the construction of the threefolds $X, Y$. As explained in \cite{IMOU} and \cite{Kuz}, the threefolds $X, Y$ are related via a diagram
  \[ \begin{array}[c]{ccccccccc}   && D & \xrightarrow{i} & M & \xleftarrow{j} & E && \\
                &&&&&&&&\\
                   &{\scriptstyle p} \swarrow \ \ && {\scriptstyle \pi_M} \swarrow \ \ \ &\downarrow & \ \ \ \searrow {\scriptstyle \rho_M} & & \ \ \searrow {\scriptstyle q} & \\
                   &&&&&&&&\\
                   X & \hookrightarrow & Q & \xleftarrow{\pi} & F & \xrightarrow{\rho}& G & \hookleftarrow & Y\\
                   \end{array}\]
                   
         Here $Q$ is a smooth $5$-dimensional quadric, and $G$ is a smooth intersection $G=\gr(2,V)\cap \PP(W)$ of a Grassmannian and a linear subspace. The morphisms $\pi$ and $\rho$ are $\PP^1$-fibrations. The morphisms $\pi_M$ and $\rho_M$ are the blow-ups with center the threefold $X$, resp. the threefold $Y$. The varieties $D, E$ are the exceptional divisors of the blow-ups.      
         
  \begin{lemma}\label{triv} Let $Q$ and $G$ be as above. We have
    \[ A^i_{hom}(Q) = A^i_{hom}(G)=0\ \ \ \forall i\ .\]
   \end{lemma}
   
   \begin{proof} It is well-known that a $5$-dimensional quadric $Q$ has trivial Chow groups. (Indeed, \cite[Corollary 2.3]{ELV} gives that $A^i_{hom}(Q)=0$ for $i\ge 3$. The Bloch--Srinivas argument \cite{BS}, combined with the fact that $H^3(Q)=0$, then implies that $A^2_{hom}(Q)=0$.)   
   
   As $\pi\colon F\to Q$ is a $\PP^1$-fibration, it follows that the variety $F$ has trivial Chow groups. But $\rho\colon F\to G$ is a $\PP^1$-fibration, and so $G$ also has trivial Chow groups.
   \end{proof} 

The blow-up formula, combined with lemma \ref{triv}, gives isomorphisms
  \[  \begin{split}  i_\ast p^\ast\colon\ \ \ A^i_{hom}(X)\ &\xrightarrow{\cong}\ A^{i+1}_{hom}(M)\ ,\\
                            j_\ast q^\ast\colon\ \ \  A^i_{hom}(Y)\ &\xrightarrow{\cong}\ A^{i+1}_{hom}(M)\ .\\     
                       \end{split}\]     
             What's more, the inverse isomorphisms are induced by a correspondence: the compositions
         \[  \begin{split}        &   A^i_{hom}(X)\ \xrightarrow{i_\ast p^\ast}\ A^{i+1}_{hom}(M)  \ \xrightarrow{ -p_\ast i^\ast}\ A^i_{hom}(X)\ ,\\
                                        &    A^i_{hom}(Y)\ \xrightarrow{j_\ast q^\ast}\ A^{i+1}_{hom}(M)  \ \xrightarrow{ -q_\ast j^\ast}\ A^i_{hom}(Y)\ ,\\
                                        &  A^{i+1}_{hom}(M)  \ \xrightarrow{ -p_\ast i^\ast}\ A^i_{hom}(X) \ \xrightarrow{  i_\ast p^\ast}\ A^{i+1}_{hom}(M) \ ,\\
                       &  A^{i+1}_{hom}(M)  \ \xrightarrow{ -q_\ast j^\ast}\ A^i_{hom}(Y) \ \xrightarrow{  j_\ast q^\ast}\ A^{i+1}_{hom}(M) \ ,\\
                                 \end{split}\]
                                 are all equal to the identity \cite[Theorem 5.3]{Vial}.       
      
    This suggests how to find a correspondence $\Gamma$ doing the job. Let us define
   \[ \Gamma:=  \Gamma_q \circ {}^t \Gamma_j \circ \Gamma_i \circ {}^t \Gamma_p\ \ \ \hbox{in}\ A^3(X\times Y)\  .\]
     
   Then we have (by the above) that
   \[ \begin{split}  &\Gamma^\ast \Gamma_\ast=\ide\colon\ \ \ A^i_{hom}(X)\ \to\    A^i_{hom}(X)\ ,\\   
                          &\Gamma_\ast \Gamma^\ast=\ide\colon\ \ \ A^i_{hom}(Y)\ \to\    A^i_{hom}(Y)\   \\
                       \end{split}\]
                       for all $i$, and so there are isomorphisms
                 \[ \Gamma_\ast\colon\ \ \ A^i_{hom}(X)\ \to\    A^i_{hom}(Y)\ \ \ \forall i\ .\]
          Given a field extension $K\supset k$, the threefolds $X_K, Y_K$ are related via a blow-up diagram as above, and so the same reasoning as above shows that there are isomorphisms
            \[ \Gamma_\ast\colon\ \ \ A^i_{hom}(X_K)\ \to\    A^i_{hom}(Y_K)\ \ \ \forall i\ .\]
        We have now established that $\Gamma$ verifies (\ref{isochow}), which clinches the proof.
     
            \end{proof}   
        

\section{A corollary}

\begin{corollary}\label{cor} Let $k$ be a finite field, and let $X,Y$ be the Calabi--Yau threefolds over $k$ constructed as in \cite{IMOU}. Then $X$ and $Y$ have the same zeta function.
\end{corollary}

\begin{proof} The zeta function can be expressed (via the Lefschetz fixed point theorem) in terms of the action of Frobenius on $\ell$-adic \'etale cohomology, hence  depends only on the motive.
\end{proof}

\begin{remark} Corollary \ref{cor} can also be deduced from \cite{Hon}, where it is proven that derived equivalent varieties of dimension $3$ have the same zeta function.
The above proof (avoiding recourse to \cite{Kuz} and \cite{Hon}) is more straightforward.
\end{remark}

\vskip1cm
\begin{nonumberingt}
This note was written at the Schiltigheim Math Research Institute. Thanks to the dedicated staff, who provide excellent working conditions.
\end{nonumberingt}

\end{document}